\documentclass[11pt,reqno]{amsart}
\usepackage{amsmath, amssymb, amsthm}

\renewcommand{\(}{\left\(}
\renewcommand{\)}{\right\)}
\renewcommand{\[}{\left\[}
\renewcommand{\]}{\right\]}

\numberwithin{equation}{section}
 \theoremstyle{plain}
\newtheorem{theorem}{Theorem}[section]
\newtheorem{lemma}[theorem]{Lemma}

\newtheorem{definition}[]{Definition}

\newtheorem{example}[]{Example}

   \makeatletter
\def\proof{\@ifnextchar[{\@oproof}{\@nproof}}
\def\@oproof[#1][#2]{\trivlist\item[\hskip\labelsep\textit{#2 \
#1.}~]\ignorespaces}
\def\@nproof{\trivlist\item[\hskip\labelsep\textit{Proof.}~]\ignorespaces}

\makeatother
\begin{document}
\title[On the combinatorics of $r$-chain minimal and maximal excludants]{On the combinatorics of $r$-chain minimal and maximal excludants}
\author{Subhash Chand Bhoria}
\address{Subhash Chand Bhoria, Department of Mathematics, Pt. Chiranji Lal Sharma Government College, Urban Estate, Sector-14, Karnal, Haryana - 132001, India.}
\email{scbhoria89@gmail.com}


\author{Pramod Eyyunni}
\address{Pramod Eyyunni, Department of Mathematics,
Birla Institute of Technology and Science Pilani, Rajasthan - 333031, India.} 
\email{pramod.eyyunni@pilani.bits-pilani.ac.in}


\author{Runqiao Li}
\address{Runqiao Li, Department of Mathematics, The Pennsylvania State University, Pennsylvania - 16801, U.S.A} 
\email{rml5856@psu.edu, runqiaoli@outlook.com}

\thanks{$2020$ \textit{Mathematics Subject Classification.} Primary 05A17, 11P84, 05A19; Secondary 11P81. \\
\textit{Keywords and phrases.} Minimal excludant, maximal excludant, $r$-chain minimal excludants, $r$-chain maximal excludants, partition identities, bijective proofs.}

\begin{abstract}
The minimal excludant (mex) of a partition was introduced by Grabner and Knopfmacher under the name `least gap' and was revived by a couple of papers due to Andrews and Newman. It has been widely studied in recent years together with the complementary partition statistic maximal excludant (maex), first introduced by Chern. Among such recent works, the first and second authors along with Maji introduced and studied the $r$-chain minimal excludants ($r$-chain mex) which led to a new generalization of Euler's classical partition theorem and the sum-of-mex identity of Andrews and Newman. In this paper, we first give combinatorial proofs for these two results on $r$-chain mex. Then we introduce the $r$-chain maximal excludants ($r$-chain maex) and establish the associated identity for the sum of $r$-chain maex over all partitions, both analytically and combinatorially.  
\end{abstract}  
\maketitle

\section{Introduction}
A \emph{partition} $\lambda$ of a positive integer $n$ is a weakly decreasing sequence of positive integers $\lambda=(\lambda_1,\lambda_2,\ldots,\lambda_{\ell})$ such that $n=\lambda_1+\lambda_2+\cdots+\lambda_{\ell}$. The terms $\lambda_i$ are called the \emph{parts} of $\lambda$, and the number of parts of $\lambda$ is called the \emph{length} of $\lambda$, denoted by $\#(\lambda)$. We let $\mathcal{P}(n)$ represent the collection of partitions of a non-negative integer $n$ and put $p(n) := |\mathcal{P}(n)|$.

In the mid eighteenth century, Euler proved the following celebrated result on partitions.

\begin{theorem}\label{Euler}
The number of partitions of a positive integer $n$ into odd parts equals the number of those into distinct parts.
\end{theorem}

This set in motion the rich area of partition identities. The reader may consult \cite{alder, pak} for more on the history of partition identities. Euler's theorem has been generalized by Glaisher \cite{glaisher} and Franklin \cite{franklin} in the following manner.

\begin{theorem}[Glaisher]\label{Glaisher} 
The number of partitions of $n$ into parts not divisible by a positive integer $r$ equals the number of partitions of $n$ with each integer occurring less than $r$ times.
\end{theorem}

\begin{theorem}[Franklin]\label{Franklin} 
The number of partitions of $n$ in which exactly $j$ different parts (these parts can be repeated) are divisible by $r$ equals the number of partitions of $n$ in which exactly $j$ different parts occur at least $r$ times each and rest of the parts appear at most $r-1$ times.
\end{theorem}

Throughout this paper, we shall call partitions with no multiple of $r$ appearing as a part as $r$ - regular partitions and those partitions where each integer appears at most $r-1$ times as $r$ - strict partitions. In the rest of the article, we adopt the following $q$-series notation: For $a, q \in \mathbb{C}$, we define
\begin{align*}
(a;q)_0 &:=1, \ (a;q)_n :=\prod_{i=0}^{n-1}(1-aq^i), \ n\in \mathbb{N}, \\
(a;q)_{\infty} &:= \lim_{n \to \infty} (a; q)_n = \prod_{i=0}^{\infty}(1-aq^i) \quad \text{for} \ |q| < 1, \\
{n\brack m}&=\frac{(q;q)_n}{(q;q)_m(q;q)_{n-m}}, \ n\geq m\geq 0.
\end{align*}
We also present below some frequently used notations throughout the sequel:
\begin{itemize}
\item $\pi$ := an integer partition,
\item $|\pi|$ := sum of the parts of $\pi$,
\item $\pi^{'}$ := the conjugate partition of $\pi$,
\item $s(\pi)$ := the smallest part of $\pi$,
\item $\ell(\pi)$ := the largest part of $\pi$
\item $\nu_{\pi}(t)$ := the multiplicity of an integer $t$ in $\pi$, the subscript being omitted if the partition in question is clear.
\item $\mathcal{P}_{r}^0(n)$ := set of partitions of $n$ where the gap between any two successive parts is at most $r$ and the smallest part is at most $r$, \label{P_0}
\item $\mathcal{P}_{r}^{+}(n)$ := $\mathcal{P}(n) \setminus \mathcal{P}_r^0(n)$, \label{P_+}
\end{itemize}

As we shall see shortly, both $\mathcal{P}_{r}^0(n)$ and $\mathcal{P}_{r}^{+}(n)$ are closely intertwined with our generalization of the maximal excludant and subsequent results. They also play a crucial role in the bijective proofs pertaining to both the minimal and maximal excludants. Before describing these results, we give a brief introduction to these concepts.  In 2019, Andrews and Newman \cite{AN19, ANII} undertook a combinatorial study of the minimal excludant (`mex') of partitions, earlier defined by Grabner and Knopfmacher \cite{grabner}. The minimal excludant of a partition $\lambda$, denoted by $mex(\lambda)$, is the smallest positive integer that does not appear in $\lambda$. They established the following result.

\begin{theorem}\label{mex1} The generating function for the sum of minimum excludants over all partitions of $n$, denoted by $\sigma\textup{mex}(n)$, is given by
$$\sum_{n=0}^{\infty} \sigma\textup{mex}(n) q^n = \sum_{n=0}^{\infty}\sum_{\lambda \in \mathcal{P}(n)} mex(\lambda)q^{n}=(-q;q)_{\infty}^{2}.$$
\end{theorem}

This theorem implies that for any positive integer $n$, the sum of the mex taken over all partitions of $n$ equals the number of partitions of $n$ into distinct parts using two colors. Later, a bijective proof of this fact was given by Ballantine and Merca \cite{BM21}.

Recently, the first and second authors along with Maji \cite{BEM22} generalized the minimal excludant and studied the $r$-chain minimal excludant of partitions. Given any partition $\lambda$, the $r$-chain minimal excludant of $\lambda$, denoted by $mex(\lambda; r)$, is the smallest positive integer $k$ such that none of $k,k+1,\ldots,k+r-1$ appears as a part of $\lambda$. They proved the following two results concerning the $r$-chain minimal excludant (`$r$-chain mex' for short). The first one below \cite[Theorem $2.4$]{BEM22} is an analogue of Franklin's result (Theorem \ref{Franklin}) concerning the total number of multiples of $r$. Here, an $r$-repeating part in a partition is an integer which has multiplicity \emph{at least} $r$.

\begin{theorem}\label{bem1}
Suppose that $j$ is a non-negative integer and $r\geq2$ is a positive integer. Then
the number of partitions of $n$ with exactly $j$ multiples of $r$ equals the number of partitions of
$n$ whose largest $r$-repeating part is $j$, which is also equal to the number of partitions of $n$ with $j$ parts greater than the $(r-1)$-chain mex. 
\end{theorem}

We provide a combinatorial proof of this identity in this paper, as requested in \cite [Problem $2$, Section $4$]{BEM22}. Coming to the second result \cite[Theorem $2.6$]{BEM22}, it is a generalization of Theorem \ref{mex1} to the case of the $r$-chain mex. 

\begin{theorem}\label{mexr3}
If $r$ is a positive integer, then
$$\sum_{n=0}^{\infty}\sum_{\lambda \in \mathcal{P}(n)} mex(\lambda; r)q^{n}=-\frac{(r-1)}{(q;q)_{\infty}}+\frac{(q^{r+1};q^{r+1})_{\infty}}{(q;q)_{\infty}}\sum_{m=1}^{r}\frac{1}{(q^m;q^{r+1})_{\infty}}.$$
\end{theorem}

We give a bijective proof of this identity as well in the sequel, answering Problem $5$ in \cite[Section $4$]{BEM22}. Instead of directly proving the theorem as stated above, we prove an equivalent statement due to a technical reason to be explained shortly. Firstly, note that Theorem \ref{mexr3} can be recast as
\begin{equation} \label{modify1}
\sum_{n=0}^{\infty}\sum_{\lambda\in\mathcal{P}(n)}\left(mex(\lambda; r)+r-1\right)q^{n}=\frac{(q^{r+1};q^{r+1})_{\infty}}{(q;q)_{\infty}}\sum_{m=1}^{r}\frac{1}{(q^m;q^{r+1})_{\infty}}.
\end{equation}
Now on the right hand side of \eqref{modify1}, the fraction of $q$-products
$$\frac{(q^{r+1};q^{r+1})_{\infty}}{(q;q)_{\infty}}$$
is the generating function for $(r+1)$ - strict partitions. Also, the sum
$$\sum_{m=1}^{r}\frac{1}{(q^m;q^{r+1})_{\infty}}$$
generates $(r+1)$ - regular partitions with all parts in the same congruence class modulo $(r+1)$. As the constant term in this sum is $r$, the empty partition must be counted for $r$ times. Before constructing a bijective proof for Theorem \ref{mexr3}, i.e. for \eqref{modify1}, we first prove a similar result that avoids counting the empty partition multiple times.

For any partition $\lambda\in\mathcal{P}(n)$, we define the following weight function:
\begin{equation} \label{omega small}
\omega_r(\lambda)=\begin{cases}0, & \lambda\in \mathcal{P}_{r}^{0}(n), \\
r-1, & \lambda\in \mathcal{P}_{r}^{+}(n).
\end{cases}
\end{equation}
We then have the following companion identity to \eqref{modify1}.
\begin{theorem}\label{mexr2}
Given any positive integer $r$, we have
$$\sum_{n=0}^{\infty}\sum_{\lambda\in\mathcal{P}(n)}(mex(\lambda; r)+\omega_r(\lambda))q^{n}=\frac{(q^{r+1};q^{r+1})_{\infty}}{(q;q)_{\infty}}\left(1+\sum_{n=1}^{\infty}\frac{q^{n}(1-q^{rn})}{(1-q^{n})(q^{r+1};q^{r+1})_{n}}\right).$$
\end{theorem}

We first obtain a bijective proof of Theorem \ref{mexr2} and then modify it to construct one for \eqref{modify1}. But the heavy lifting is done in the proof of Theorem \ref{mexr2}. 

Inspired by the minimal excludant, Chern \cite{C21} investigated the maximal excludant of partitions (`maex' for short). He defined the maximal excludant of a partition $\lambda$, denoted by $maex(\lambda)$, to be the largest integer less than the largest part of $\lambda$ that does not appear in $\lambda$. He proved the following result for the maex as an analogue of Theorem \ref{mex1}. Here, $\sigma L (n)$ denotes the sum of largest parts taken over all partitions of $n$, i.e., $\sum_{\pi \in \mathcal{P}(n)} \ell(\pi)$.

\begin{theorem}\label{max1} The sum of maximum excludants of all partitions of $n$, denoted by $\sigma\textup{maex}(n)$, satisfies the following identity:
$$\sum_{n=0}^{\infty}\left(\sigma L(n) - \sigma\textup{maex}(n) \right)q^{n}=\frac{1}{(q;q)_{\infty}}\sum_{n=1}^{\infty}q^n(q^{2};q^{2})_{n-1}.$$
\end{theorem}

It is natural to consider a generalization of the maex analogous to the $r$-chain minimal excludant. 
\begin{definition} [$r$-chain maximal excludant] \label{r-chain maex}
For $r\geq1$, we define the $r$-chain maximal excludant (`$r$-chain maex' for short) of a partition $\lambda$, denoted by $maex(\lambda; r)$, to be the largest positive integer $k$ less than the largest part of $\lambda$ such that the \textbf{positive integers} $k,k-1,\ldots,k-r+1$ do not occur as parts in $\lambda$. If there is no such positive integer $k$ for $\lambda$, we set $maex(\lambda; r)=0$.
\end{definition}
From this definition, $maex(\lambda; r) > 0$ if and only if there is a string of at least $r$ consecutive integers missing between two successive parts of $\lambda$ or if the smallest part of $\lambda$ is greater than $r$ or both. This means precisely that $ \lambda \in \mathcal{P}_{r}^{+}(n)$, as defined in the frequently used notations just before the statement of Theorem \ref{mex1}. This also implies that $maex(\lambda; r) = 0$ if and only if $\lambda \in \mathcal{P}(n) \setminus \mathcal{P}_{r}^{+}(n) = \mathcal{P}_{r}^{0}(n)$. We then have the following theorem for the $(r-1)$-chain maex analogous to Theorem \ref{bem1}.

\begin{theorem}\label{bem2}
Suppose that $r\geq2$ and $j\geq1$ are two positive integers. Then the number
of partitions of $n$ where the largest multiple of $r$ occurs exactly $j$ times equals the number of partitions of $n$ whose smallest $r$-repeating part is $j$. Moreover, both these quantities also equal the number of partitions in $\mathcal{P}_{r-1}^+(n)$ with $j$ parts greater than the $(r-1)$-chain maex. 
\end{theorem}

\begin{example}
Let $r=3$ and $j=1$ and consider the partitions in $\mathcal{P}_{2}^{+}(7)$, i.e. the partitions of $7$ whose $2$-chain maex is positive with exactly one part greater than it. One can check that the partitions $7, \ 6+1, \ 5+2, \ 5+1+1$ and $4+1+1+1$ are precisely the ones in this category. Now, coming to the partitions of $7$ whose smallest $3$-repeating part is $1$, we see that $4+1+1+1, \ 3+1+1+1+1, \ 2+2+1+1+1, \ 2+1+1+1+1+1$ and $1+1+1+1+1+1+1$ are the required ones. Finally, partitions of $7$ whose largest multiple of $3$ occurs exactly once are $6+1, \ 4+3, \ 3+2+2, \ 3+2+1+1$ and $3+1+1+1+1$. In all the cases, there are an equal number of partitions (exactly five) as expected.
\end{example}

For a partition $\lambda\in\mathcal{P}(n)$, we define another weight function:
\begin{equation} \label{omega capital}
\Omega_r(\lambda)=\begin{cases}1, & \lambda\in \mathcal{P}_{r}^{0}(n), \\
r, & \lambda\in \mathcal{P}_{r}^{+}(n).
\end{cases}
\end{equation}

Then we have the following generalization of Theorem \ref{max1}, which can be considered to be a maximal excludant analogue for Theorem \ref{mexr3}.

\begin{theorem}\label{maxr1}
For any positive integer $r$, we have
\begin{equation} \label{maxr1 id}
\sum_{n=0}^{\infty}\sum_{\lambda\in\mathcal{P}(n)}(\ell(\lambda) - maex(\lambda; r)+\Omega_{r}(\lambda))q^{n}=\frac{(q^{r+1};q^{r+1})_{\infty}}{(q;q)_{\infty}}
+\frac{1}{(q;q)_{\infty}}\sum_{n=1}^{\infty}\frac{q^n(q^{r+1};q^{r+1})_{n}}{1-q^n}.
\end{equation}
\end{theorem}

The rest of this paper is organized as follows. In Section \ref{bijective proofs}, we give combinatorial proofs for Theorem \ref{bem1} and Theorem \ref{bem2}, along with a $q$-series proof for Theorem \ref{bem2}. In Section \ref{Identity on the sum of $r$-chain minimal excludant}, we demonstrate Theorem \ref{mexr2} and Theorem \ref{mexr3} combinatorially. Finally, in Section \ref{The sum of $r$-chain maximal excludants}, we prove Theorem \ref{maxr1} both analytically and bijectively.

\section{Proofs of Theorem \ref{bem1} and Theorem \ref{bem2}} \label{bijective proofs}
We first give a combinatorial proof of Theorem \ref{bem1} and then proceed to
prove Theorem \ref{bem2} both analytically and combinatorially.

\begin{proof}[Combinatorial proof of Theorem \ref{bem1}][]

Firstly, note that partitions with $j$ parts greater than the $(r-1)$-chain mex on conjugation give rise to partitions whose largest $r$-repeating part is $j$. Now, we give a bijective proof of the equinumerosity of partitions having $j$ multiples of $r$ and those with largest $r$-repeating part $j$. We denote by $E_{j, r}(n)$ and $L_{j, r}(n)$, respectively, the collection of partitions of $n$ into $j$ multiples of $r$ and the collection whose largest $r$-repeating part is $j$. Before establishing a bijection between them, we recall a bijection due to Glaisher \cite{glaisher} between the set $\mathcal{O}_r(n)$ of partitions of an integer $n$ into parts not divisible by $r$ and the set $\mathcal{D}_r(n)$ of partitions of $n$ where each integer repeats less than $r$ times. Let 
\begin{equation} \label{f map} 
f: \mathcal{O}_r(n) \rightarrow \mathcal{D}_r(n)
\end{equation}
be the map acting on a partition $\pi \in \mathcal{O}_r(n)$ by merging every $r$ copies of the same integer until there is no integer which repeats $r$ times or more. And in the reverse direction, for a partition in $\mathcal{D}_r(n)$, we split each multiple of $r$ into $r$ equal parts until there is no multiple of $r$ remaining. It is readily seen that this process is the inverse of $f$ and so $f$ is a bijection. Next, we define the `concatenation' of two partitions $\alpha, \mu$ in $\mathcal{P}$ (the two partitions need not partition the same integer $m$) to be another partition $\alpha * \mu$ obtained by taking together all the parts in $\alpha$ and $\mu$. For example, if $\alpha = (3^2, 1^3)$ and $\mu = (5, 2^3, 1^2)$, then $\alpha * \mu = (5, 3^2, 2^3, 1^5)$.

We now construct a bijective map $\phi_r: E_{j, r}(n) \rightarrow L_{j, r}(n)$. A partition $\pi$ in $E_{j, r}(n)$ can be written uniquely as $\pi_o * \pi_r$, with $\pi_r$ consisting of all the parts that are multiples of $r$ and the remaining parts being included in $\pi_o$. Since $\pi \in E_{j, r}(n)$, it has $j$ multiples of $r$ and so $\pi_r = (r\ell_1, r\ell_2, \dots, r\ell_j)$ for some $\ell_1 \geq \ell_2 \geq \cdots \geq \ell_j \geq 1$. Now put $\tilde{\pi_o}:= f(\pi_o)$, where $f$ is the map defined in \eqref{f map}. Then $\Tilde{\pi_o}$ is a partition where each integer appears less than $r$ times. Further, define $\tilde{\pi_r} := \pi_r^{'}$, and observe that $\tilde{\pi_r}$ has largest part equal to $j$. Moreover, each integer in $\tilde{\pi_r}$ repeats a multiple of $r$ times. In other words, $j$ appears at least $r$ times in $\tilde{\pi_r}$. Finally, noting that $\tilde{\pi_o} * \tilde{\pi_r}$ is a partition with largest $r$-repeating part $j$, we set $\phi_r(\pi) := \tilde{\pi_o} * \tilde{\pi_r}$.
We next define a map $\psi_r: L_{j, r}(n) \rightarrow E_{j, r}(n)$, which will turn out to be the inverse of $\phi_r$. Suppose that $\nu \in L_{j, r}(n)$, so that its largest $r$-repeating part is $j$. Write $\nu = (m^{f_m}, \dots, j^{f_j}, \dots, 1^{f_1})$ where $\ell(\nu) = m, \ f_j \geq r,  \ f_t \geq 0$ for $t<j$ and $0 \leq f_t < r$ for $t>j$. Employing the division algorithm, we have $f_i = rq_i + s_i$, where $0 \leq s_i < r$ for each $i$. Next, define $\nu_r := ( j^{rq_j}, \dots, 1^{rq_1})$ and $\nu_o := (m^{f_m}, \dots, j^{s_j}, \dots, 1^{s_1})$. Note that since $0 \leq s_i < r$ and $0 \leq f_t < r$ for $t>j$, $\nu_o$ is a partition where each integer appears less than $r$ times. Next, set $\kappa := f^{-1} (\nu_o)$, which is an $r$-regular partition, and $ \gamma := \nu_r^{'}$, which is a partition with parts that are multiples of $r$ such that $\#(\gamma) = j$. As a result, $\kappa * \gamma \in E_{j, r}(n)$ and we put $\psi_r(\nu) := \kappa * \gamma$. It is easy to check that $\psi_r$ is indeed the inverse of $\phi_r$.
\end{proof}

\begin{example}
This is an illustration in the case when $r=3$. Take $\pi := (9, 7, 6^3, 1^4)$. Note that there are $j=4$ multiples of $r$. Split this into $\pi_r := (9, 6^3)$ and $\pi_o := (7, 1^4)$. Then $\tilde{\pi_o} := f(\pi_o) = (7, 3, 1)$ and $\tilde{\pi_r} := \pi_r^{'} = (4^6, 1^3)$. Therefore, $\phi_r(\pi) = \tilde{\pi_o} * \tilde{\pi_r} = (7, 4^6, 3, 1^4)$, which has its largest $3$-repeating part equal to $4$. Conversely, beginning with $\nu = (7, 4^6, 3, 1^4)$, we see that $\nu_r = (4^6, 1^3)$ and $\nu_o = (7, 3, 1)$. Now, $\kappa := f^{-1} (\nu_o) = (7, 1^4)$ and $\gamma := \nu_r^{'} = (9, 6^3)$ and thereby $\psi_r(\nu) = \kappa * \gamma = (9, 7, 6^3, 1^4) = \pi$.  
\end{example}

Before proceeding to the proof of Theorem \ref{bem2} and beyond, we note down a crucial ingredient, namely, the $q$-binomial theorem: For $|q| < 1$ and any $a \in \mathbb{C}$, we have
\begin{equation}\label{q-binomial}
\sum_{n=0}^{\infty} \frac{(a; q)_n}{(q; q)_n} z^n = \frac{(az; q)_{\infty}}{(z; q)_{\infty}}.
\end{equation}

\begin{proof} [Analytic proof of Theorem \ref{bem2}][]
The proof proceeds in two parts. We first prove that the partitions in $\mathcal{P}_{r-1}^{+}(n)$ having $j$ parts greater than their $(r-1)$-chain maex are equinumerous with partitions of $n$ whose smallest $r$-repeating part is $j$. 

\textbf{Part 1}: Suppose $k \geq 1$ is the $(r-1)$-chain maex of a partition $\pi$ and that there are $j$ parts greater than $k$. This means that $k, \ k-1, \ \cdots, k-r+2$ are all positive integers and missing from $\pi$. First of all, since $k-r+2 \geq 1$, we have $k \geq r-1$. Note that $\pi$ comprises two pieces, the first piece $\pi_1$ consisting of parts less than $k-r+2$ with no restriction on what and how many parts may occur. So the generating function for this piece is 
\begin{equation} \label{GFn first piece}
\frac{1}{(q; q)_{k-r+1}}.
\end{equation}
Coming to the second piece $\pi_2$ of $\pi$, it consists of those parts of $\pi$ that are greater than $k$. Note that $k + 1$ must be a part, for otherwise it becomes the $(r-1)$-chain maex of $\pi$. Also, parts in $\pi_2$ must occur in such a manner so that no string of $r-1$ consecutive integers should be missing (to ensure that $k$ is the $(r-1)$-chain maex of $\pi$). That means in $\pi_2$, the gaps between successive parts must be at most $r-1$. Moreover, since there are $j$ parts in $\pi$ greater than the $(r-1)$-chain maex $k$, this translates to $\pi_2$ having exactly $j$ parts. So, in summary, $s(\pi_2) = k+1, \ \#(\pi) = j$, and gaps between successive parts $\leq r-1$. By conjugation, this implies $\nu(\ell(\pi_2^{'})) = k+1, \ \ell(\pi_2^{'}) = j$ and that each integer in $\pi_2^{'}$ repeats less than $r$ times. The generating function for such partitions (and hence for those of the type $\pi_2$) is:
\begin{equation} \label{GFn second piece}
q^{j(k+1)} \prod_{n=1}^{j-1} (1 + q^n + \cdots + q^{n(r-1)}).
\end{equation}
Putting together the generating functions for the two pieces $\pi_1$ and $\pi_2$, i.e., \eqref{GFn first piece} and \eqref{GFn second piece} and summing over the possible values of $k$, we can write the generating function for partitions $\pi$ as follows:
\begin{align} 
\sum_{k=r-1}^{\infty} \frac{q^{j(k+1)}}{(q; q)_{k-r+1}} \prod_{n=1}^{j-1} (1 + q^n + \cdots + q^{n(r-1)})  &= \prod_{n=1}^{j-1} (1 + q^n + \cdots + q^{n(r-1)}) \sum_{k=r-1}^{\infty} \frac{q^{j(k+1)}}{(q; q)_{k-r+1}} \nonumber \\
&= \prod_{n=1}^{j-1} (1 + q^n + \cdots + q^{n(r-1)}) \sum_{k=0}^{\infty} q^{jr} \frac{q^{jk}}{(q; q)_k} \nonumber \\
&= \frac{q^{jr}}{(q^j; q)_{\infty}} \prod_{n=1}^{j-1} (1 + q^n + \cdots + q^{n(r-1)}) \nonumber \\
&= \frac{q^{jr}}{1- q^j} \times \frac{1}{(q^{j+1}; q)_{\infty}} \prod_{n=1}^{j-1} (1 + q^n + \cdots + q^{n(r-1)}), \label{j parts >}
\end{align}
where we invoked \eqref{q-binomial} in the penultimate equality. Now, the expression in \eqref{j parts >}
is precisely the generating function for partitions with smallest $r$-repeating part $j$.

\textbf{Part 2}: Coming to the second part of the theorem, we prove that the expression in \eqref{j parts >} also generates partitions with the largest multiple of $r$ occurring exactly $j$ times. We write
{\allowdisplaybreaks
\begin{align}
\frac{q^{jr}}{1- q^j} \times \frac{1}{(q^{j+1}; q)_{\infty}} \prod_{n=1}^{j-1} (1 + q^n + \cdots + q^{n(r-1)}) &= \frac{q^{jr}}{(q^j; q)_{\infty}} \times \frac{(q^r; q^r)_{j-1}}{(q; q)_{j-1}} \nonumber \\
&= \frac{q^{jr}}{(q; q)_{\infty}} \times \frac{(q^r; q^r)_{\infty}}{(q^{jr}; q^r)_{\infty}} \nonumber \\
&= \frac{q^{jr} (q^r; q^r)_{\infty}}{(q; q)_{\infty}} \sum_{k=0}^{\infty} \frac{(q^{jr})^k}{(q^r; q^r)_k}  \quad \text{(by \eqref{q-binomial})} \nonumber \\
&= \frac{(q^r; q^r)_{\infty}}{(q; q)_{\infty}} \sum_{k=1}^{\infty} \frac{q^{krj}}{(q^r; q^r)_{k-1}} \nonumber \\
&= \frac{(q^r; q^r)_{\infty}}{(q; q)_{\infty}} \sum_{k=1}^{\infty} \frac{1}{1 - q^r} \times  \dots \times \frac{1}{1 - q^{(k-1)r}} \times (q^{kr})^j, \nonumber
\end{align}}
and the last expression is precisely the generating function for partitions with $j$ occurrences of the largest multiple of $r$.  
\end{proof}

\begin{proof} [Combinatorial proof of Theorem \ref{bem2}][]
The proof closely parallels the proof of Theorem \ref{bem1} and so, we give only a sketch of the ideas involved. Firstly, note that partitions in $\mathcal{P}_{r-1}^+(n)$ with $j$ parts greater than the $(r-1)$-chain maex on conjugation give rise to unrestricted partitions whose smallest $r$-repeating part is $j$. Let $H_{j, r}(n)$ and $S_{j, r}(n)$ denote the collection of partitions of $n$ in which the largest multiple of $r$ has multiplicity $j$ and that collection whose smallest $r$-repeating part is $j$ respectively. 
Take any partition $\pi$ in $H_{j, r}(n)$, then it can be expressed uniquely in the form $\pi_o * \pi_r$, with $\pi_r$ containing only multiples of $r$ and all the remaining parts in $\pi_o$. As $\pi \in H_{j, r}(n)$, it has $j$ copies of the largest multiple of $r$. We can write it as $\pi_r = (\underbrace{r\ell_1, \dots, r\ell_1}_{j \ \text{times}}, r\ell_2, \dots, r\ell_j)$ for some $\ell_1 > \ell_2 \geq \cdots \geq \ell_j \geq 1$. Next, set $\tilde{\pi_o}:= f(\pi_o)$, where $f$ is Glaisher's bijection, as defined in the proof of Theorem \ref{bem1}. Further, we also define $\tilde{\pi_r} := \pi_r^{'}$, and note that $\tilde{\pi_r}$ is a partition where each integer has multiplicity a multiple of $r$ with smallest part being $j$. So, $\tilde{\pi_o} * \tilde{\pi_r}$ is a partition whose smallest $r$-repeating part is $j$, and we define $\Phi_r: H_{j, r}(n) \to S_{j, r}(n)$ by $\Phi_r(\pi) := \tilde{\pi_o} * \tilde{\pi_r}$.

In the other direction, suppose that $\nu \in S_{j, r}(n)$, so that the smallest $r$-repeating part of $\nu$ equals $j$. Then $\nu$ can be represented in the form $\nu = (m^{f_m}, \dots, j^{f_j}, \dots, 1^{f_1})$ with $\ell(\nu) = m, \ f_j \geq r,  \ f_k \geq 0$ for $k > j$ and $0 \leq f_k < r$ for $k < j$. Using the division algorithm, we can write $f_i = rq_i + s_i$, where $0 \leq s_i < r$ for each $i$. We now define $\nu_r := ( m^{rq_m}, \dots, j^{rq_j})$ and $\nu_o := (m^{s_m}, \dots, j^{s_j}, (j-1)^{f_{j-1}}, \dots, 1^{f_1})$. We can see that $\nu_o$ is a partition with each integer appearing less than $r$ times and hence $\kappa := f^{-1} (\nu_o)$ is an $r$-regular partition. Also, $ \gamma := \nu_r^{'}$ is a partition with parts that are multiples of $r$ such that the largest part has multiplicity $j$. Therefore, $\kappa * \gamma \in H_{j, r}(n)$ and setting $\Psi_r(\nu) := \kappa * \gamma$, we see that $\Psi_r$ is the inverse of $\Phi_r$. In this way, we have a constructed a bijection between $H_{j, r}(n)$ and $S_{j, r}(n)$ and the proof is complete.
\end{proof}

\section{Bijective proofs of identities for the sum of $r$-chain minimal excludants} \label{Identity on the sum of $r$-chain minimal excludant}

In this section, we give a combinatorial proof of Theorem \ref{mexr3}. Before going to that, we first prove Theorem \ref{mexr2}. 


\begin{proof} [Proof of Theorem \ref{mexr2}][]
Note that from the definition of the weight $\omega_r$ in \eqref{omega small}, we have
\begin{equation} \label{modify2}
\sum_{n=0}^{\infty}\sum_{\lambda\in\mathcal{P}(n)}(mex(\lambda; r)+\omega_r(\lambda))q^{n}=\sum_{n=0}^{\infty}\sum_{\lambda\in\mathcal{P}(n)}(mex(\lambda; r)+r-1)q^{n}-\sum_{n=0}^{\infty}\sum_{\lambda\in \mathcal{P}_{r}^{0}(n)}(r-1)q^{n},
\end{equation}
and the first sum on the right hand side is given by \eqref{modify1}. So it remains to calculate the second sum on the right hand side. We have
\begin{align*}
\sum_{n=0}^{\infty}\sum_{\lambda\in \mathcal{P}_{r}^{0}(n)}(r-1)q^{n}
=(r-1)\sum_{n=0}^{\infty}| \mathcal{P}_{r}^{0}(n)|q^{n} = (r-1)\frac{(q^{r+1};q^{r+1})_{\infty}}{(q;q)_{\infty}}.
\end{align*}
And this used along with the right hand side of \eqref{modify1} turns \eqref{modify2} into 
\begin{equation} \label{modify3}
\sum_{n=0}^{\infty}\sum_{\lambda\in\mathcal{P}(n)}(mex(\lambda; r)+\omega_r(\lambda))q^{n}=\frac{(q^{r+1};q^{r+1})_{\infty}}{(q;q)_{\infty}}\sum_{m=1}^{r}\frac{1}{(q^m;q^{r+1})_{\infty}}-(r-1)\frac{(q^{r+1};q^{r+1})_{\infty}}{(q;q)_{\infty}}.
\end{equation}
Now for any integer $1\leq m\leq r$, we have
$$
\frac{1}{(q^m;q^{r+1})_{\infty}}
=\sum_{n=0}^{\infty}\frac{q^{nm}}{(q^{r+1};q^{r+1})_{n}}
=1+\sum_{n=1}^{\infty}\frac{q^{nm}}{(q^{r+1};q^{r+1})_{n}},
$$
thus giving us
$$\sum_{m=1}^{r}\frac{1}{(q^m;q^{r+1})_{\infty}}=\sum_{m=1}^{r}\left(1+\sum_{n=1}^{\infty}\frac{q^{nm}}{(q^{r+1};q^{r+1})_{n}}\right)=r+\sum_{n=1}^{\infty}\frac{q^{n}(1-q^{rn})}{(1-q^{n})(q^{r+1};q^{r+1})_{n}},$$
and consequently
$$\frac{(q^{r+1};q^{r+1})_{\infty}}{(q;q)_{\infty}}\sum_{m=1}^{r}\frac{1}{(q^m;q^{r+1})_{\infty}}
=\frac{(q^{r+1};q^{r+1})_{\infty}}{(q;q)_{\infty}}\left(r+\sum_{n=1}^{\infty}\frac{q^{n}(1-q^{rn})}{(1-q^{n})(q^{r+1};q^{r+1})_{n}}\right).$$
Therefore, from \eqref{modify3}, we finally have
$$\sum_{n=0}^{\infty}\sum_{\lambda\in\mathcal{P}(n)}(mex(\lambda; r)+\omega_r(\lambda))q^{n}=
\frac{(q^{r+1};q^{r+1})_{\infty}}{(q;q)_{\infty}}\left(1+\sum_{n=1}^{\infty}\frac{q^{n}(1-q^{rn})}{(1-q^{n})(q^{r+1};q^{r+1})_{n}}\right).$$
\end{proof}
We will first construct a bijection for Theorem \ref{mexr2} and then extend it to prove \eqref{modify1}. For this, we interpret Theorem \ref{mexr2} as counting certain objects.

Let $S_r(n)$ be the number of $r$ - strict partitions of $n$. Let $M_r(n)$ be the number of partitions of $n$ such that the multiplicity of the largest part is not divisible by $r$ while all the other parts must have a multiplicity that is divisible by $r$. Then we have
\begin{align*}
\sum_{n=0}^{\infty}S_{r+1}(n)q^{n}&=\frac{(q^{r+1};q^{r+1})_{\infty}}{(q;q)_{\infty}},
\end{align*}
and
\begin{align}
\sum_{n=0}^{\infty} M_{r+1}(n)q^{n}
&= 1+\sum_{n=1}^{\infty}\left(\frac{q^n}{1-q^n}-\frac{q^{(r+1)n}}{1-q^{(r+1)n}}\right) \prod_{i=1}^{n-1}(1+q^{(r+1)i}+q^{2(r+1)i}+\cdots) \nonumber \\
&= 1 + \sum_{n=1}^{\infty} \frac{q^n (1 - q^{rn})}{(1 - q^n) (1-q^{(r+1)n})} \times \frac{1}{(q^{r+1};q^{r+1})_{n-1}} \nonumber \\
&= 1+\sum_{n=1}^{\infty}\frac{q^{n}(1-q^{rn})}{(1-q^{n})(q^{r+1};q^{r+1})_{n}}. \nonumber
\end{align}

So in Theorem \ref{mexr2}, the coefficient of $q^{n}$ on the left hand side is the number of ordered pairs $(\lambda, i)$ such that
\begin{equation} \label{lambda, i}
\lambda \in \mathcal{P}(n), \quad \text{and} \quad 
\begin{cases} 1\leq i \leq mex(\lambda; r), & \text{if} \quad \lambda\in \mathcal{P}_{r}^{0}(n), \\
1\leq i\leq mex(\lambda; r)+r-1, & \text{if} \quad \lambda\in \mathcal{P}_{r}^{+}(n).
\end{cases}
\end{equation}
On the other hand, the coefficient of $q^{n}$ on the right hand side in Theorem \ref{mexr2} is the number of ordered pairs of partitions $(\alpha,\beta)$ such that
\begin{align}
& |\alpha| + |\beta| = n, \quad \text{$\alpha$ is an $(r+1)$ - strict partition}, \nonumber \\ 
& \text{$\nu_{\beta}(\ell(\beta))$ is not divisible by $r+1$, \quad \text{and} \quad  $\nu_{\beta}(t)$ is divisible by $r+1$ for $t < \ell(\beta)$}. \label{alpha, beta}
\end{align}
So now our aim is to construct a bijection between the pairs of partitions in \eqref{lambda, i} and \eqref{alpha, beta}.

We first introduce some operators on partitions.
\begin{itemize}
    \item \textbf{Conjugate}

For any partition $\lambda=(\lambda_1,\ldots,\lambda_\ell)$, let $\lambda'$ denote the conjugate of $\lambda$.
    \item \textbf{The horizontal cut}

For any integer $1\leq i\leq \ell+1$, we define two partitions $\lambda\uparrow^{(i)}$ and $\lambda\downarrow^{(i)}$ as
$$\lambda\uparrow^{(i)}=(\lambda_1,\lambda_2,\ldots,\lambda_{i-1})\ \text{and}\ \lambda\downarrow^{(i)}=(\lambda_i,\lambda_{i+1},\ldots,\lambda_\ell),$$
with $\lambda\uparrow^{(1)}=\lambda\downarrow^{(\ell + 1)}=\emptyset$. Note that in view of the Ferrers diagram for a partition, this means that the Ferrers diagram of $\lambda$ can be cut into two parts horizontally between $\lambda_{i-1}$ and $\lambda_i$. Then $\lambda\uparrow^{(i)}$ and $\lambda\downarrow^{(i)}$ are the partitions corresponding to the upper part and lower part respectively.

    \item \textbf{The $\Psi$ operator}

Let $(\alpha,\beta)$ be a pair of partitions. We are going to define an operator $\Psi$ on such pairs. To this end, we write $\beta$ in a more compact notation:
$$\beta=\left(\beta_{1}^{\nu(\beta_1)},\beta_{2}^{\nu(\beta_2)},\ldots,\beta_{k}^{\nu(\beta_k)}\right),$$
where $\beta_1>\beta_2>\cdots>\beta_{k}$. Then for each $i$ we can find unique nonnegative integers $q_i,h_i$ such that
$$\nu(\beta_i)=q_i(r+1)+h_i,$$
with $r$ being a fixed positive integer and $0 \leq h_i \leq r$. Now we retain all the copies of $\beta_1$ in $\beta$, and for each $2\leq i\leq k$, we move $h_i$ copies of $\beta_i$ from $\beta$ to $\alpha$. For convenience, let the resultant new versions of $\alpha$ and $\beta$, by abuse of notation, be denoted by $\Psi(\alpha)$ and $\Psi(\beta)$ respectively. Then we define $\Psi(\alpha,\beta) :=(\Psi(\alpha),\Psi(\beta))$.
\end{itemize}

Based on the definition of the operator $\Psi$, the following properties are immediate.
\begin{lemma}\label{a1}
If the smallest part of $\alpha$ is greater than or equal to the largest part of $\beta$, and $\alpha$ is an $(r+1)$-strict partition. Then $\Psi(\alpha)$ is also an $(r+1)$-strict partition.
\end{lemma}
\begin{lemma}\label{a2}
In $\Psi(\beta)$, only the largest part may have multiplicity not divisible by $r+1$.
\end{lemma}

Now we are ready to construct a  bijection $\Gamma$ for Theorem \ref{mexr2}.

\begin{proof}[Bijective proof of Theorem \ref{mexr2}][]
Given any partition $\lambda$ and an integer $i$ with $1\leq i\leq mex(\lambda; r)+\omega_r(\lambda)$, we consider the following cases.

\textbf{Case 1.} Let $(\lambda,i)$ be an ordered pair such that $\lambda\in \mathcal{P}_{r}^{0}(n)$ and $1\leq i\leq mex(\lambda; r)$.

In this case, let $\lambda'=(\lambda'_1,\lambda'_2, \dots )$ be the conjugate of $\lambda$. We apply the operator $\Psi$ on $(\lambda'\uparrow^{(i)},\lambda'\downarrow^{(i)})$, with the claim that $\Psi(\lambda'\uparrow^{(i)},\lambda'\downarrow^{(i)})$ is of the type described in \eqref{alpha, beta}. Indeed, by the definition of $\mathcal{P}_{r}^{0}(n)$, we have $\lambda_{j}-\lambda_{j+1}\leq r $ for all $j$ less than the length of $\lambda$, and the smallest part of $\lambda$ is also less than or equal to $r$. Thus in the conjugate $\lambda'$, each part appears at most $r$ times. So this is also true in $\lambda'\uparrow^{(i)}$. Now by Lemma \ref{a1}, it follows that each part in $\Psi(\lambda'\uparrow^{(i)})$ also
repeats at most $r$ times. As for $\Psi(\lambda'\downarrow^{(i)})$, we only need to show that the multiplicity of its largest part is not a mutiple of $(r+1)$. And this is indeed the case since the largest part of $\Psi(\lambda'\downarrow^{(i)})$ is nothing but the largest part of $\lambda'\downarrow^{(i)}$, which has multiplicity less than $r+1$. We thus set $\Gamma(\lambda,i)=\Psi(\lambda'\uparrow^{(i)},\lambda'\downarrow^{(i)})$.

\textbf{Case 2.} Let $(\lambda,i)$ be an ordered pair such that 
$\lambda\in \mathcal{P}_{r}^{+}(n)$ and $1\leq i\leq mex(\lambda; r)-1$.

In this case, the mapping is the same as that in Case $1$. By the definition of $mex(\lambda; r)$, one can observe that $\lambda'\uparrow^{(i)}$ is still an $(r+1)$ - strict partition. And the largest part of $\lambda'\downarrow^{(i)}$ still has multiplicity less than $r+1$. Thus $\Psi(\lambda'\uparrow^{(i)},\lambda'\downarrow^{(i)})$ is once again of the desired type. So we again let $\Gamma(\lambda,i)=\Psi(\lambda'\uparrow^{(i)},\lambda'\downarrow^{(i)})$.

\textbf{Case 3.} Let $(\lambda,i)$ be an ordered pair such that 
$\lambda\in \mathcal{P}_{r}^{+}(n)$ and $mex(\lambda; r)\leq i\leq mex(\lambda; r)+r-1$.

In this case, the first step is still to find $\Psi(\lambda'\uparrow^{(i)},\lambda'\downarrow^{(i)})$. Now let $G(\lambda)$ be the number of parts in $\lambda$ that are greater than $mex(\lambda; r)$ and let $k$ be the multiplicity of $G(\lambda)$ in $\lambda'$. Then $k \geq (r+1)$ and  for $mex(\lambda; r)\leq i\leq mex(\lambda; r)+k-1$, we have $\lambda'_i=G(\lambda)$. So this time the largest part of $\lambda'\downarrow^{(i)}$, which will be equal to $G(\lambda)$, has  multiplicity $k-(i-mex(\lambda; r))$ in $\lambda'\downarrow^{(i)}$. And since this may or may not be a multiple of $(r+1)$, we need to consider two subcases as follows.

\textbf{Subcase 3.1.} If $k-(i-mex(\lambda; r))\not\equiv 0\ (\text{mod}\ (r+1))$, that means the multiplicity of the largest part of $\Psi(\lambda'\downarrow^{(i)})$ is not divisible by $(r+1)$. So $\Psi(\lambda'\uparrow^{(i)},\lambda'\downarrow^{(i)})$ is exactly what we need. Thus, set $\Gamma(\lambda,i)=\Psi(\lambda'\uparrow^{(i)},\lambda'\downarrow^{(i)})$ in this subcase.

\textbf{Subcase 3.2.} If $k-(i-mex(\lambda; r))\equiv 0\ (\text{mod}\ (r+1))$, then we move $r-(i-mex(\lambda; r))$ copies of $G(\lambda)$ from $\Psi(\lambda'\downarrow^{(i)})$ to $\Psi(\lambda'\uparrow^{(i)})$. Let the modified versions of $\Psi(\lambda'\downarrow^{(i)})$ and $\Psi(\lambda'\uparrow^{(i)})$ be denoted by $\Psi^*(\lambda'\downarrow^{(i)})$ and $\Psi^*(\lambda'\uparrow^{(i)})$ respectively. Now we have $r$ copies of $G(\lambda)$ in $\Psi^*(\lambda'\uparrow^{(i)})$ and $k-r$ copies of $G(\lambda)$ in $\Psi^*(\lambda'\downarrow^{(i)})$. Since $0 \leq i-mex(\lambda; r) \leq r-1$, we see that $1 \leq r-(i-mex(\lambda; r)) \leq r$, and so $k - r = \left(k - (i-mex(\lambda; r))\right) - \left(r-(i-mex(\lambda; r))\right)$ is not divisible by $(r+1)$. Hence, we put $\Gamma(\lambda,i)= \left(\Psi^*(\lambda'\uparrow^{(i)}), \Psi^*(\lambda'\downarrow^{(i)}) \right)$ for this subcase. This finishes the construction of the mapping $\Gamma$. 

We have seen that $\Gamma(\lambda,i)$ satisfies the constraints on the pair $(\alpha,\beta)$ in \eqref{alpha, beta} in all the cases. Now we need to show that it is invertible. Suppose that we are given a pair of partitions $(\alpha,\beta)$ such that $\alpha$ is an $(r+1)$ - strict partition, and in $\beta$, the multiplicity of the largest part is not divisible by $(r+1)$, while the other parts have multiplicities divisible by $(r+1)$. Then we wish to recover the pair $(\lambda,i)$ such that
$$|\lambda|=|\alpha|+|\beta|, \quad 1\leq i\leq mex(\lambda; r)+\omega_r(\lambda),\ \text{and} \quad \Gamma(\lambda, i) = (\alpha,\beta).$$ 
It is straightforward to find the corresponding $\lambda$, we simply need to put all the parts of $\alpha$ and $\beta$ together, rearrange them in the weakly decreasing order and take the conjugate. And now with $\lambda$ in hand, we can find $G(\lambda)$ and its multiplicity in $\lambda'$, let it be $k$. 

If $G(\lambda)$ has multiplicity $r$ in $\alpha$, then this would correspond to Subcase $3.2.$, and we know that $i-mex(\lambda; r)\equiv k\ (\text{mod}\ (r+1))$. Now since $mex(\lambda; r)\leq i\leq mex(\lambda; r)+r-1$ and $k \not \equiv r \ (\text{mod}\ (r+1))$, we can determine $i$ uniquely.

On the other hand, suppose the multiplicity of $G(\lambda)$ in $\alpha$ is different from $r$. Then it would correspond to one of the other cases. But note that in each of those cases, since $\lambda'\uparrow^{(i)}$ consists of parts greater than or equal to the largest part of $\lambda'\downarrow^{(i)}$, applying the operator $\Psi$ simply moves some (possibly none) of the parts in $\lambda'\downarrow^{(i)}$ that are less than the largest part of $\lambda'\downarrow^{(i)}$ to $\lambda'\uparrow^{(i)}$ giving us $\Psi(\lambda'\uparrow^{(i)},\lambda'\downarrow^{(i)})$. There is no transfer of the largest parts in $\lambda'\downarrow^{(i)}$ to $\lambda'\uparrow^{(i)}$. In other words, $\ell(\lambda'\downarrow^{(i)}) = \ell(\Psi(\lambda'\downarrow^{(i)})) = \ell(\beta)$. So to get $i$, we just need to count the number of parts in $\alpha$ that are greater than or equal to the largest part of $\beta$ and add $1$ to it. This concludes the bijective proof of the theorem.
\end{proof}

Next we extend the definition of $\Gamma$ to $\Gamma^*$, which gives a bijective proof for Theorem \ref{mexr3}.

\begin{proof}[Bijective proof of \eqref{modify1}][]
Note that if $(\alpha,\beta)$ is a pair of partitions counted by the right hand side of Theorem \ref{mexr2}, then $(\alpha,\beta')$ will also be a pair counted by the right hand side of \eqref{modify1}. So the only difference is that the pairs $(\alpha;\beta)$ must be counted multiple times when $\beta$ is the empty partition. To this end, we assume that the empty partition can be assigned any of $r$ different colors. We will denote by $(\alpha;\emptyset_{i})$ the pair with arbitrary $\alpha$ and $\beta$ being the empty partition and  assigned the $i$-th color (for $1 \leq i \leq r$).

Now, given a partition $\lambda$, we consider two different cases. If $\lambda\in \mathcal{P}_{r}^{+}(n)$, then $mex(\lambda; r)+r-1$ is less than $\ell(\lambda)$, i.e., the largest part of $\lambda$. And our bijection $\Gamma$ works fine on them, so we take $\Gamma^*(\lambda; i)=\Gamma(\lambda, i)$. Also, in this case, if $(\alpha;\beta)$ is the image of $\lambda$ under our map, then $\beta$ will never be empty. Since the length of the conjugate of $\lambda$, which is just $\ell(\lambda)$, is greater than $mex(\lambda; r)+r-1$, there will always be some parts assigned to $\beta$.

For the partitions $\lambda\in \mathcal{P}_{r}^{0}(n)$, we have $mex(\lambda; r)=\ell(\lambda)+1$. In this case, $\lambda'$, the conjugate of $\lambda$, has length $\ell(\lambda)$, which means we can only cut it horizontally in $\ell(\lambda)+1$ different ways. In this case, we fail to find the other $(r-1)$ images of $\lambda$. But this can be easily fixed as we can assign different colors to $\beta$ when it is empty. Suppose we have a pair $(\lambda; i)$ where $\lambda$ is a partition in $\mathcal{P}_{r}^{0}(n)$ while $1\leq i\leq mex(\lambda; r)+r-1$ is an integer. If $1\leq i\leq mex(\lambda; r)-1=\ell(\lambda)$, then we can apply our map $\Gamma$ as usual and $\beta$ will not be empty, so we define $\Gamma^*(\lambda;i)=\Gamma(\lambda;i)$ as well. And for $mex(\lambda; r)\leq i\leq mex(\lambda; r)+r-1$, we define the image as $\Gamma^*(\lambda; i)= (\lambda',\emptyset_{i-mex(\lambda; r)+1})$. The inverse of this additional definition is clear. So this finishes the proof.
\end{proof}

\section{The sum of $r$-chain maximal excludants} \label{The sum of $r$-chain maximal excludants}

In this section, we prove Theorem \ref{maxr1}.

\begin{proof}[Analytic proof of Theorem \ref{maxr1}][]
The left hand side of \eqref{maxr1 id} can be written as
\begin{equation} \label{2 sums}
\sum_{n=0}^{\infty}\sum_{\lambda\in\mathcal{P}(n)}\ell(\lambda) q^{n}-\sum_{n=0}^{\infty}\sum_{\lambda\in\mathcal{P}(n)}(maex(\lambda; r)-\Omega_r(\lambda))q^{n}
\end{equation}
and we treat the two double sums separately. 

For the first double sum, we proceed as follows.
\begin{align}
\sum_{n=0}^{\infty}\sum_{\lambda\in\mathcal{P}(n)}\ell(\lambda) q^{n}
=\left.\frac{\partial}{\partial{z}}\sum_{n=0}^{\infty}\frac{z^{n}q^{n}}{(q;q)_{n}}\right|_{z=1} &= \left.\frac{\partial}{\partial{z}}\frac{1}{(zq;q)_{\infty}}\right|_{z=1} \nonumber \\
&=\frac{1}{(zq;q)_{\infty}}\left.\frac{\partial}{\partial{z}}\log\frac{1}{(zq;q)_{\infty}}\right|_{z=1} \nonumber \\
&=\frac{1}{(zq;q)_{\infty}}\left.\frac{\partial}{\partial{z}}\sum_{n=1}^{\infty}-\log(1-zq^n)\right|_{z=1} \nonumber \\
&=\frac{1}{(q;q)_{\infty}}\sum_{n=1}^{\infty}\frac{q^n}{1-q^n}. \label{ell lambda}
\end{align}

For the second double sum in \eqref{2 sums}, as required by the definitions of $maex_{r}(\lambda)$ and $\Omega_r(\lambda)$, we divide the partitions of $n$ into two classes. Doing this gives us
\begin{equation} \label{2 further sums}
\sum_{n=0}^{\infty}\sum_{\lambda\in\mathcal{P}(n)}(maex(\lambda; r)-\Omega_r(\lambda))q^{n}=\sum_{n=0}^{\infty}\sum_{\lambda\in \mathcal{P}_{r}^{0}(n)}(maex(\lambda; r)-\Omega_r(\lambda))q^{n}+\sum_{n=0}^{\infty}\sum_{\lambda\in \mathcal{P}_{r}^{+}(n)}(maex(\lambda; r)-\Omega_r(\lambda))q^{n}.
\end{equation}

For the partitions $\lambda$ in $\mathcal{P}_{r}^{0}(n)$, we have $maex(\lambda; r)-\Omega_r(\lambda)=0-1=-1$. Thus,
\begin{equation} \label{inter inter}
\sum_{n=0}^{\infty}\sum_{\lambda\in \mathcal{P}_{r}^{0}(n)}(maex(\lambda; r)-\Omega_r(\lambda))q^{n}
=\sum_{n=0}^{\infty}\sum_{\lambda\in \mathcal{P}_{r}^{0}(n)}(-1)q^{n}
=-\sum_{n=0}^{\infty}| \mathcal{P}_{r}^{0}(n)|q^{n}
=-\frac{(q^{r+1};q^{r+1})_{\infty}}{(q;q)_{\infty}},
\end{equation}
where the last equality arises by taking the conjugates of the partitions in $\mathcal{P}_{r}^{0}(n)$.

Let $\mathcal{P}_{r}^{+}(m,n)$ be the subcollection of partitions in $\mathcal{P}_{r}^{+}(n)$ with $maex(\lambda; r)=m$. We count these partitions according to the number $\ell$ of parts that are greater than $m$ in $\lambda$. We have
\begin{align}
\mathbb{P}_{r}^+(z,q) := \sum_{n=0}^{\infty}\sum_{m=r}^{\infty}|\mathcal{P}_{r}^{+}(m,n)|z^mq^n
&=\sum_{m=r}^{\infty}\sum_{\ell=1}^{\infty}\frac{z^mq^{(m+1)\ell}}{(q;q)_{m-r}}\frac{(q^{r+1};q^{r+1})_{\ell-1}}{(q;q)_{\ell-1}} \nonumber \\
&=\sum_{m=0}^{\infty}\sum_{\ell=1}^{\infty}\frac{z^{m+r}q^{(m+r+1)\ell}}{(q;q)_{m}}\frac{(q^{r+1};q^{r+1})_{\ell-1}}{(q;q)_{\ell-1}} \nonumber \\
&=\sum_{\ell=1}^{\infty}\frac{(q^{r+1};q^{r+1})_{\ell-1}}{(q;q)_{\ell-1}}z^{r}q^{(r+1)\ell}
\sum_{m=0}^{\infty}\frac{z^{m}q^{m\ell}}{(q;q)_{m}} \nonumber \\
&=\sum_{\ell=1}^{\infty}\frac{(q^{r+1};q^{r+1})_{\ell-1}}{(q;q)_{\ell-1}}z^{r}q^{(r+1)\ell}
\frac{1}{(zq^{\ell};q)_{\infty}} \nonumber \\
&=z^{r}\sum_{n=0}^{\infty}\frac{q^{(r+1)(n+1)}(q^{r+1};q^{r+1})_{n}}{(q;q)_n(zq^{n+1};q)_{\infty}}. \label{Przq}
\end{align}

Note that
\begin{equation} \label{another 2 sums}
\sum_{n=0}^{\infty}\sum_{\lambda\in \mathcal{P}_{r}^+(n)}(maex(\lambda; r)-\Omega_r(\lambda))q^{n}=\left.\frac{\partial}{\partial{z}}\mathbb{P}_{r}^+(z,q)\right|_{z=1}-r\mathbb{P}_{r}^+(1,q).
\end{equation}
Now we calculate the two terms in the right hand side above. Firstly, by \eqref{Przq}, observe that
\begin{align}
\left.\frac{\partial}{\partial{z}}\mathbb{P}_{r}^+(z, q)\right|_{z=1}
&=\left.\frac{\partial}{\partial{z}}z^{r}\sum_{n=0}^{\infty}\frac{q^{(r+1)(n+1)}(q^{r+1};q^{r+1})_{n}}{(q;q)_n(zq^{n+1};q)_{\infty}}\right|_{z=1} \nonumber \\
&=r\mathbb{P}_{r}^+(1, q)+\left.\frac{\partial}{\partial{z}}\sum_{n=0}^{\infty}\frac{q^{(r+1)(n+1)}(q^{r+1};q^{r+1})_{n}}{(q;q)_n(zq^{n+1};q)_{\infty}}\right|_{z=1}. \label{inter przq}
\end{align}
Thus, from \eqref{Przq} and \eqref{inter przq}, we have 
{\allowdisplaybreaks
\begin{align}
\sum_{n=0}^{\infty}\sum_{\lambda\in \mathcal{P}_{r}^{+}(n)}(maex(\lambda; r)-\Omega_r(\lambda))q^{n}
&=\left.\frac{\partial}{\partial{z}}\sum_{n=0}^{\infty}\frac{q^{(r+1)(n+1)}(q^{r+1};q^{r+1})_{n}}{(q;q)_n(zq^{n+1};q)_{\infty}}\right|_{z=1} \nonumber \\
&=\sum_{n=0}^{\infty}\frac{q^{(r+1)(n+1)}(q^{r+1};q^{r+1})_{n}}{(q;q)_n}\left.\frac{\partial}{\partial{z}}\frac{1}{(zq^{n+1};q)_{\infty}}\right|_{z=1} \nonumber \\
&=\sum_{n=0}^{\infty}\frac{q^{(r+1)(n+1)}(q^{r+1};q^{r+1})_{n}}{(q;q)_n(zq^{n+1};q)_{\infty}}\left.\frac{\partial}{\partial{z}}\log\frac{1}{(zq^{n+1};q)_{\infty}}\right|_{z=1} \nonumber \\
&=\sum_{n=0}^{\infty}\frac{q^{(r+1)(n+1)}(q^{r+1};q^{r+1})_{n}}{(q;q)_n(zq^{n+1};q)_{\infty}}\sum_{m=1}^{\infty}\left.\frac{\partial}{\partial{z}} \left\{-\log(1-zq^{n+m})\right\} \right|_{z=1} \nonumber \\
&=\sum_{n=0}^{\infty}\frac{q^{(r+1)(n+1)}(q^{r+1};q^{r+1})_{n}}{(q;q)_n(q^{n+1};q)_{\infty}}\sum_{m=1}^{\infty}\frac{q^{n+m}}{1-q^{n+m}} \nonumber \\
&=\frac{1}{(q;q)_{\infty}}\sum_{n=0}^{\infty}q^{(r+1)(n+1)}(q^{r+1};q^{r+1})_{n}\sum_{m=n+1}^{\infty}\frac{q^{m}}{1-q^{m}} \nonumber \\
&=\frac{1}{(q;q)_{\infty}}\sum_{m=1}^{\infty}\frac{q^{m}}{1-q^{m}}\sum_{n=0}^{m-1}q^{(r+1)(n+1)}(q^{r+1};q^{r+1})_{n} \nonumber \\
&=\frac{1}{(q;q)_{\infty}}\sum_{m=1}^{\infty}\frac{q^{m}}{1-q^{m}}(1-(q^{r+1};q^{r+1})_{m}). \label{final inter}
\end{align}}
Finally, using \eqref{ell lambda}, \eqref{2 further sums}, \eqref{inter inter}, and \eqref{final inter}, we obtain
{\allowdisplaybreaks \begin{align*}
\sum_{n=0}^{\infty}\sum_{\lambda\in\mathcal{P}(n)}(\ell(\lambda)-maex(\lambda; r)+\Omega_r(\lambda))q^{n}
=&\sum_{n=0}^{\infty}\sum_{\lambda\in\mathcal{P}(n)}\ell(\lambda) q^{n}-\sum_{n=0}^{\infty}\sum_{\lambda\in\mathcal{P}(n)}(maex(\lambda; r)-\Omega_r(\lambda))q^{n}\\
=&\sum_{n=0}^{\infty}\sum_{\lambda\in\mathcal{P}(n)}\ell(\lambda) q^{n}-\left(\sum_{n=0}^{\infty}\sum_{\lambda\in \mathcal{P}_{r}^{0}(n)}(maex(\lambda; r)-\Omega_r(\lambda))q^{n}\right.\\
&\left.+\sum_{n=0}^{\infty}\sum_{\lambda\in \mathcal{P}_{r}^{+}(n)}(maex(\lambda; r)-\Omega_r(\lambda))q^{n}\right)\\
=&\frac{1}{(q;q)_{\infty}}\sum_{n=1}^{\infty}\frac{q^n}{1-q^n}-\left(-\frac{(q^{r+1};q^{r+1})_{\infty}}{(q;q)_{\infty}}\right.\\
&+\left.\frac{1}{(q;q)_{\infty}}\sum_{m=1}^{\infty}\frac{q^{m}}{1-q^{m}}(1-(q^{r+1};q^{r+1})_{m})\right)\\
=&\frac{1}{(q;q)_{\infty}}\sum_{n=1}^{\infty}\frac{q^n}{1-q^n}+
\frac{(q^{r+1};q^{r+1})_{\infty}}{(q;q)_{\infty}}\\
&-\frac{1}{(q;q)_{\infty}}\sum_{m=1}^{\infty}\frac{q^{m}}{1-q^{m}}(1-(q^{r+1};q^{r+1})_{m})\\
=&\frac{(q^{r+1};q^{r+1})_{\infty}}{(q;q)_{\infty}}+\frac{1}{(q;q)_{\infty}}\sum_{m=1}^{\infty}\frac{q^{m}}{1-q^{m}}(q^{r+1};q^{r+1})_{m}.
\end{align*}}
This finishes the proof.
\end{proof}

Next we aim to give a combinatorial proof for Theorem \ref{maxr1}, which will be similar to the bijection $\Gamma$ for Theorem \ref{mexr2}. We first need to interpret the right hand side of Theorem \ref{maxr1} as the generating function of certain partition-theoretic objects. Recall that
\begin{equation}\label{midmid1}
\sum_{n=0}^{\infty}S_{r+1}(n)q^{n} = \frac{(q^{r+1};q^{r+1})_{\infty}}{(q;q)_{\infty}}
\end{equation}
is the generating function for $(r+1)$ - strict partitions. Let $W_{r}(n)$ be the number of partitions of $n$ such that only the smallest part may have multiplicity not congruent to $0$ mod $r$. Then we have
\begin{align}
\sum_{n=0}^{\infty}W_{r+1}(n)q^{n}&=1+\sum_{n=1}^{\infty}\frac{q^n}{1-q^n}
\prod_{i=1}^{\infty}(1+q^{(r+1)(n+i)}+q^{2(r+1)(n+i)}+\cdots) \nonumber \\
&=1+\sum_{n=1}^{\infty}\frac{q^n}{1-q^n}\frac{1}{(q^{(n+1)(r+1)};q^{r+1})_{\infty}}\nonumber \\
&=1+\frac{1}{(q^{r+1};q^{r+1})_{\infty}}\sum_{n=1}^{\infty}\frac{q^n(q^{r+1};q^{r+1})_{n}}{1-q^n}. \label{midmid2}
\end{align}
By \eqref{midmid1} and \eqref{midmid2}, we finally get
\begin{equation} \label{alternate rhs T1.10}
\left(\sum_{n=0}^{\infty}S_{r+1}(n)q^{n}\right)
\left(\sum_{n=0}^{\infty}W_{r+1}(n)q^{n}\right)=
\frac{(q^{r+1};q^{r+1})_{\infty}}{(q;q)_{\infty}}
+\frac{1}{(q;q)_{\infty}}\sum_{n=1}^{\infty}\frac{q^n(q^{r+1};q^{r+1})_{n}}{1-q^n},
\end{equation}
which is the right hand side of Theorem \ref{maxr1}.

Now we are ready to give the bijective proof. The coefficient of $q^{n}$ on the left hand side of Theorem \ref{maxr1} is the number of ordered pairs $(\lambda, i)$ such that
\begin{equation} \label{maex lambda, i}
\lambda \in \mathcal{P}(n), \quad \text{and} \quad 1\leq i\leq \ell(\lambda) -maex(\lambda; r)+\Omega_r(\lambda).
\end{equation}
Coming to the right hand side of Theorem \ref{maxr1} and taking into account its alternate form in \eqref{alternate rhs T1.10}, i.e., $\left(\sum_{n=0}^{\infty}S_{r+1}(n)q^{n}\right)
\left(\sum_{n=0}^{\infty}W_{r+1}(n)q^{n}\right)$, we see that the coefficient of $q^n$ here is the number of ordered pairs of partitions $(\alpha,\beta)$ such that
\begin{align}
& |\alpha| + |\beta| = n, \quad \text{$\alpha$ is an $(r+1)$ - strict partition}, \nonumber \\ 
& \text{$\nu_{\beta}(t)$ is divisible by $r+1$ for $t > s(\beta)$}. \label{maex alpha, beta}
\end{align}
So we aim to construct a bijection between the pairs of partitions in \eqref{maex lambda, i} and \eqref{maex alpha, beta}. This time, we will need an operator similar to $\Psi$.
%
%

We define an operator $\Phi$ on pairs $(\alpha, \beta)$ of partitions. Firstly, write $\beta$ in a more compact notation:
$$\beta= \left(\beta_{1}^{\nu(\beta_1)},\beta_{2}^{\nu(\beta_2)},\ldots,\beta_{k}^{\nu(\beta_k)}\right),$$
with $\beta_1>\beta_2>\cdots>\beta_{k}$. Then for each $i$ we can find unique nonnegative integers $q_i,h_i$ such that
$$\nu(\beta_i)=q_i(r+1)+h_i,$$
where $r$ is a fixed positive integer and $0 \leq h_i\leq r$. Now we keep all the copies of $\beta_k$ in $\beta$, and for each $1\leq i\leq k-1$, we move $h_i$ copies of $\beta_i$ from $\beta$ to $\alpha$. For the sake of convenience let the resulting versions of $\alpha$ and $\beta$, by abuse of notation, be represented by $\Phi(\alpha)$ and $\Phi(\beta)$ respectively. Then we define $\Phi(\alpha,\beta) :=(\Phi(\alpha),\Phi(\beta))$.
\begin{proof}[Bijective proof of Theorem \ref{maxr1}][]
For any $(\lambda,i)$ with $\lambda \in \mathcal{P}(n)$ and $1\leq i\leq \ell(\lambda) -maex(\lambda; r)+\Omega_r(\lambda)$, define
$$\Delta(\lambda,i)=\Phi(\lambda'\downarrow^{(\ell(\lambda)+2-i)},\lambda'\uparrow^{(\ell(\lambda)+2-i)}).$$
We show that this gives a bijective proof for Theorem \ref{maxr1}.

Firstly, we need to show that $\Delta$ is well defined. For this, all we need to check is $1\leq \ell(\lambda)+2-i\leq \ell(\lambda)+1$, since otherwise, the horizontal cut cannot be performed on $\lambda'$. The second inequality $\ell(\lambda)+2-i\leq \ell(\lambda)+1$ readily follows as $i \geq 1$. Now if $\lambda\in\mathcal{P}_{r}^{0}(n)$, then $\Omega_r(\lambda)=1$ and $maex(\lambda; r) = 0$, so that $1\leq i\leq \ell(\lambda)+1$. This gives us $\ell(\lambda)+2 \geq i + 1$ or $\ell(\lambda)+2 - i \geq 1$. On the other hand, if $\lambda\in \mathcal{P}_{r}^{+}(n)$, then we have $\Omega_r(\lambda)=r$ and $r\leq maex(\lambda; r)<\ell(\lambda)$, so that
$$\ell(\lambda)+2-i\geq \ell(\lambda)+2- \ell(\lambda)+maex(\lambda; r)-\Omega_r(\lambda) = 2 + maex(\lambda; r) - r \geq 2.$$
So in both the cases, $\Delta$ is well defined.

Next, we need to show that $\Phi(\lambda'\downarrow^{(\ell(\lambda)+2-i)},\lambda'\uparrow^{(\ell(\lambda)+2-i)})$ is of the type described in \eqref{maex alpha, beta}. By our choice of $i$ and the definition of $maex(\lambda; r)$, one can see that $\lambda'\downarrow^{(\ell(\lambda)+2-i)}$ is always an $(r+1)$ - strict partition. So $\Phi(\lambda'\downarrow^{(\ell(\lambda)+2-i)})$ is an $(r+1)$ - strict partition as well. And from the definition of the operator $\Phi$, it is clear that for $\Phi(\lambda'\uparrow^{(\ell(\lambda)+2-i)})$, only its smallest part may have multiplicity not divisible by $r+1$. So $\Phi(\lambda'\downarrow^{(\ell(\lambda)+2-i)},\lambda'\uparrow^{(\ell(\lambda)+2-i)})$ satisfies all the conditions that we are looking for.

Finally, we need to show that $\Delta$ is invertible. Given any pair $(\alpha,\beta)$, we can first recover $\lambda$ in the same manner as that of the inverse of $\Gamma$. And now one can see that $i$ is equal to $1$ plus the number of parts in $\alpha$ which are less than or equal to the smallest part of $\beta$. Here we take the smallest part of $\beta$ to be infinity when $\beta$ is the empty partition. This concludes the proof.
\end{proof}

\section{Acknowledgements}
The first author expresses his thanks towards his institution Pt. CLS Govt. College, Sector-14, Karnal, Haryana under the aegis of Directorate of Higher Education Haryana (India) for their immense support and for providing a conducive research environment. The second author wishes to thank his current employers BITS Pilani and former employers IIT Indore for the wonderful research facilities. The third author is thankful to Prof. Ae Ja Yee for fruitful discussions and for providing helpful suggestions. In addition, the authors of this work are grateful to Dr. Damanvir Singh Binner for initial discussions thereby providing impetus to this work.


\begin{thebibliography}{99}

\bibitem{alder}
H.~L.~Alder, Partition identities -- From Euler to the present, \emph{Amer. Math. Monthly}, Issue 7, \textbf{76} (1969), 733--746.

\bibitem{AN19}
G.~E.~Andrews and D.~Newman, Partitions and the minimal excludant, \emph{Ann. Comb.}, \textbf{23(2)} (2019), 249--254.

\bibitem{ANII}
G.~E.~Andrews and D.~Newman, The minimal excludant in integer partitions, \emph{J. Integer Seq.}, \textbf{23} (2020), Article 20.2.3. 

\bibitem{BM21}
C.~Ballantine and M.~Merca, Combinatorial proof of the minimal excludant theorem, \emph{Int. J. Number Theory}, \textbf{17(08)} (2021), 1765--1779.

\bibitem{BEM22}
S.~C.~Bhoria, P.~Eyyunni and B.~Maji, A new generalization of the minimal excludant arising from an analogue of Franklin's Identity, \emph{Discrete Math.}, Issue 5, \textbf{346} (2023), 113334.

\bibitem{C21}
S.~Chern, Partitions and the maximal excludant, \emph{Electron. J. Combin.},  \textbf{28(3)} (2021), Article P3.13.

\bibitem{franklin}
F.~Franklin, On partitions, \emph{Johns Hopkins Univ. Cir.}, \textbf{2} (1883), p. 72.

\bibitem{glaisher} 
J.~W.~L.~Glaisher, A theorem in partitions, \emph{Messenger of Math.}, \textbf{12} (1883), 158--170.

\bibitem{grabner}
 P.~J.~Grabner and A.~Knopfmacher, Analysis of some partition statistics, \emph{Ramanujan J.}, \textbf{12} (2006), 439--454.

\bibitem{pak}
I.~Pak, Partition bijections, a survey, \emph{Ramanujan J.}, \textbf{12} (2006), 5--75.

\end{thebibliography}
\end{document}